%% file: ms.tex
\documentclass[preprint]{elsarticle}

\usepackage{geometry}
\usepackage{graphicx}
\usepackage{caption}
\usepackage{subcaption}
\usepackage{amsmath}
\usepackage{amssymb}
\usepackage{amsthm}
\usepackage{tikz}
\usetikzlibrary{graphs}
\usepackage{algorithm}
\usepackage{algpseudocode}
\usepackage{siunitx}
\usepackage{hyperref}
\usepackage{placeins}

\newcommand{\R}{\mathbb{R}}

\newcommand{\norm}[1]{\left\lVert#1\right\rVert}
\newcommand{\sos}{\mathbb{P}[x]}
\newcommand{\ssos}{\mathbb{P}_+[x]}

\newcommand{\comment}[1]{\textcolor{purple}{#1}}

\newtheorem{thm}{Theorem}
\newtheorem{dfn}{Definition}
\newtheorem{lemma}{Lemma}
\newtheorem*{remark}{Remark}

\title{Stable Sparse Operator Inference for Nonlinear Structural Dynamics}

\author{Pascal den Boef\corref{cor1}\textsuperscript{\emph{a},}}
\ead{p.d.boef@tue.nl}
\author{Diana Manvelyan\textsuperscript{\emph{b}}}
\ead{diana.manvelyan@siemens.com}
\author{Joseph Maubach\textsuperscript{\emph{a}}}
\ead{j.m.l.maubach@tue.nl}
\author{Wil Schilders\textsuperscript{\emph{a}}}
\ead{w.h.l.schilders@tue.nl}
\author{Nathan van de Wouw\textsuperscript{\emph{a}}}
\ead{n.v.d.wouw@tue.nl}

\cortext[cor1]{Corresponding author}

\affiliation[1]{organization={Eindhoven University of Technology},
	city={Eindhoven},
	country={The Netherlands}}

\affiliation[2]{organization={Siemens AG},
	city={Garching},
	country={Germany}}

\begin{document}

\begin{abstract}
	Structural dynamics models with nonlinear stiffness appear, for example, when analyzing systems with nonlinear material behavior or undergoing large deformations.
	For complex systems, these models become too large for real-time applications or multi-query workflows.
	Hence, model reduction is needed.
	However, the mathematical operators of these models are often not available since, as is common in industry practice, the models are constructed using commercial simulation software.
	In this work, we propose an operator inference-based approach aimed at inferring, from data generated by the simulation model, reduced-order models (ROMs) of structural dynamics systems with stiffness terms represented by polynomials of arbitrary degree.
	To ensure physically meaningful models, we impose constraints on the inference such that the model is guaranteed to exhibit stability properties.
	Convexity of the optimization problem associated with the inference is maintained by applying a sum-of-squares relaxation to the polynomial term.
	To further reduce the size of the ROM and improve numerical conditioning of the inference, we also propose a novel clustering-based sparsification of the polynomial term.
	We validate the proposed method on several numerical examples, including a representative 3D Finite Element Model (FEM) of a steel piston rod.
\end{abstract}

\begin{keyword}
	model order reduction \sep operator inference \sep sparsity \sep stability \sep digital twin technology
\end{keyword}	

\maketitle

\section{Introduction}
\input{intro}

\section{Stabilized Operator Inference}
\label{sct:stabopinf}
\input{stabopinf}

\section{Hyperreduction with Sparse Monomial Basis}
\label{sct:hyperreduction}
\input{hyperreduction}

\FloatBarrier

\section{Numerical examples}
\label{sct:examples}
\input{examples}

\section{Conclusion}
\label{sct:conclusion}
\input{conclusion}

\bibliographystyle{elsarticle-num}
\bibliography{references}

\end{document}

%% file: intro.tex
In recent years, \emph{digital twins} (DT) have emerged as an important tool towards digitalization in industry.
They provide computer-aided assistance for real-world models allowing monitoring of the physical counterparts\cite{benner2020model, hartmann2018model, li2021digital}.
By collecting, storing, analyzing, and providing feedback on data, a DT enables continuous evaluation of its physical entity.
Sensors capture and continually update the DT throughout its operational life, allowing engineers and operators to maintain a live view of the physical system.
Many applications can benefit from digital twins having structural mechanics as their foundation.
The required online models are often based on partial differential equations which in turn are converted to a large number of ordinary differential equations (ODE) using methods such as finite element methods.
In this work, the type of ODE systems we consider are nonlinear structural dynamics models of the following second-order form:
\begin{equation}
	M_y \ddot{y}(t) + C_y \dot{y}(t) + f_y(y(t)) = B_y u(t),
	\label{eqn:model_f_fom}
\end{equation}
where $M_y$, $C_y$ and $B_y$ are appropriately sized matrices.
$y(t) \in \mathbb{R}^{n}$ is the displacement and $u(t) \in \mathbb{R}^{n_u}$ the input.
The function $f_y: \mathbb{R}^{n} \rightarrow \mathbb{R}^{n}$ models the stiffness force term and, depending on the application, can be nonlinear.
Such models arise, for example, when modelling deformation of nonlinear materials such as rubber, or when considering large deformations \cite{martynov2019polynomial}.

Simulating high-fidelity models of the form (\ref{eqn:model_f_fom}) can be extremely time consuming \cite{hartmann2018model, hartmann202012}.
Model order reduction (MOR) is a pivotal technology for such models enabling real-time execution of digital twins.
Moreover, MOR enables the reusability of such models for other scenarios. 
The principle of MOR is to reduce the number of degrees of freedom $n$ while keeping the essential physical properties of the model.
This is possible since, even for large $n$, the displacement $y(t)$ often evolves on a relatively low-dimensional manifold.
This low-dimensional manifold may be represented by a basis $V \in \mathbb{R}^{n \times r}$ with $r \ll n$ and corresponding coordinates $x(t) \in \mathbb{R}^{r}$ that approximate the displacement field as follows:
\begin{equation}
	y(t) \approx V x(t).
	\label{eqn:y=Vx}
\end{equation}
The dynamic model describing the evolution of $x(t)$ is
\begin{equation}
	M \ddot{x}(t) + C \dot{x}(t) + f(x(t)) = B u(t),
	\label{eqn:model_f}
\end{equation}
with $M, C, f$ and $B$ the reduced counterparts of the operators in (\ref{eqn:model_f_fom}).

In the linear case, there are well-established model reduction methods to find reduced operators of (\ref{eqn:model_f}) to approximate (\ref{eqn:model_f_fom}).
Many of these methods are \emph{intrusive}: that is, they require access to the operators of the ODE system (\ref{eqn:model_f_fom}).
For example, Guyan reduction \cite{guyan1965reduction} is one of the earliest methods for computing a reduced-order model that can accurately match the behavior of the full-order model for inputs that vary sufficiently slow (relative to the eigenfrequencies of the system).
Guyan reduction is a modal truncation method (in which only static modes are considered).
Other modal truncation methods have been developed afterwards, such as the Craig-Bampton method \cite{craig1968coupling}.
While modal truncation methods are applied in the context of structural dynamics (which is also the focus of this paper), many other model reduction methods have been developed in different communities.
For example, moment matching \cite{gragg1983partial, boley1994krylov} in the numerical mathematics community and balanced truncation \cite{moore1981principal} in the systems theory community.
For a comparison of these linear model reduction techniques from different backgrounds, see \cite{gallivan1999model, besselink2013comparison}.

The linear model reduction techniques discussed so far require a mathematical description of the full-order model.
However, often only snapshot data of the dynamics are available while the data-generating model itself is not.
For example, an explicit mathematical model of the dynamics is unavailable if commercial finite element analysis (FEA) software is used.
To address the challenge that these operators are unavailable, \emph{non-intrusive} techniques have also been developed.
For example, in \cite[Section 5.2]{filanova2023operator} an operator inference method is proposed for linear second-order models with guaranteed stability.
Dynamic Mode Decomposition (DMD) is another technique introduced in \cite{schmid2010dynamic} for linear dynamic system models and approximates the dominant eigenmodes by application of the Arnoldi method.

In literature, several methods are proposed to extend the data-driven compact modelling techniques to nonlinear systems.
DMD is generalized to nonlinear dynamics via the Koopman operator in \cite{rowley2009spectral}. 
The method of SINDy \cite{brunton2016discovering} is another method that aims at discovering models from data.
In SINDy, the basic assumption is that the nonlinear term appearing in the dynamics equations (e.g., $f_y$ in (\ref{eqn:model_f_fom})), can be written as a linear combination of a set of basis functions.
This set of basis functions may be large if little prior knowledge of the dynamics is available.
SINDy then attempts to find a sparse selection from this set of basis functions such that the resulting dynamics closely match the provided data.
The sparse selection is achieved by using an additional sparsity-promoting regularizing term in the optimization.
However, stability of the resulting model is not guaranteed.
A variant of SINDy that guarantees global stability for systems with quadratic nonlinearities is proposed in \cite{kaptanoglu2021promoting}.
Herein, the ``trapping theorem" from \cite{schlegel2015long} is used.
This theorem consists of necessary and sufficient conditions for a dynamical system to possess a bounded trapping region to which all trajectories converge.
Similarly, in \cite{sawant2023physics} a regularizing term is added that promotes stability for models with function $f_y$ consisting of quadratic and cubic nonlinear terms.
A limitation of the above methods, which we seek to address, is the restriction to quadratic nonlinearities.
In the method proposed in the current paper, we will enable handling nonlinear polynomial terms with arbitrary degree.

An alternative way of stability preservation is obtained by considering model structures that inherently preserve stability.
For first-order systems with quadratic nonlinearity, the authors of \cite{goyal2023guaranteed} develop operator inference methods based on Lyapunov analysis.
They derive a structured parametrization of the quadratic term that allows them to certify stability using a Lyapunov function.
Based on this principle, they give 3 separate operator inference methods that each enforces a different notion of stability: 1) local stability; 2) global stability; and 3) an attracting trapping region.
However, the resulting optimization problem that needs to be solved to infer the operators is non-convex.
The operator inference method we propose in this work is based on convex optimization, which is computationally advantageous.
Another work that uses energy-based arguments to guarantee stability is \cite{sharma2024preserving}, where the Lagrangian structure of second-order systems is considered.
Such structure then automatically implies stability of the dynamics.
Similarly, the work in \cite{sharma2022hamiltonian} preserves the Hamiltonian structure of a model.
However, the methods based on such a structure either rely on precise knowledge of the energy function or are restricted to special cases (such as quadratic energy functions).
In our method, no such prior knowledge is required, nor do such restrictions apply.
We consider a more general form of the nonlinearity $f$ consisting of polynomials with arbitrary degree.
Furthermore, we will provide global stability properties, as opposed to stability on a bounded subset of state-space.
For an overview of recent work on learning models from data using operator inference, consult the review paper \cite{kramer2024learning}.

Another data-driven approach to discover stable models from data is \cite{kolter2019learning} using deep neural networks to model dynamics.
In this method, a deep neural network is jointly learned with a Lyapunov function.
The deep neural network function $h(x)$ describes the dynamics according to $\dot{x} = \mathrm{Proj}(h(x))$, where $\mathrm{Proj}$ is a projection operator that is governed by the Lyapunov function such that stability is guaranteed.
The authors demonstrate the capabilities of the method by learning the stable video texture of a burning flame.
Although deep learning neural networks are effective tools in these types of settings (high-dimensional image data and little prior knowledge on the governing physics), selecting the appropriate model architecture can be challenging.
This challenge is further increased because the authors use a nonlinear Variational Autoencoder to provide dimensionality reduction from $y(t)$ to $x(t)$.
This is in contrast to most operator inference techniques, which use Proper Orthogonal Decomposition (POD) \cite{berkooz1993proper,lu2019review} to find a linear projection from $y(t)$ to $x(t)$ as in (\ref{eqn:y=Vx}).
In addition to the challenge posed by selecting an interpretable and effective architecture, the optimization problem to be solved during the learning process is typically non-convex and the learned Lyapunov function has no physical interpretability.

The authors of \cite{martynov2019polynomial} also treat model reduction for nonlinear models of the form (\ref{eqn:model_f_fom}).
Specifically, the authors focus on the case that the elements of $f_y(y(t))$ are monomials of $y(t)$ up to degree 3.
After an initial reduction of the dynamics using POD, the elements of $f(x(t))$ in (\ref{eqn:model_f}) are again monomials of $x(t)$ up to degree 3.
However, the authors note that the number of monomials grows rapidly even for moderate $r$, 
This rapid growth limits the computational scalability of such a model structure.
To resolve this issue, the authors propose to use a sparse selection of monomial terms.
The selection is made in a data-driven manner based on an ordering of the monomials computed from the snapshot data.
In this way, the conditioning and scalability of the problem are improved.
Inspired by this, we also introduce a method to incorporate sparsity in the models.
We design an algorithm that performs the sparse selection in such a way that it is compatible with the stability constraints we aim to enforce.
Furthermore, the numerical experiments in \cite{martynov2019polynomial} demonstrate that, for some configurations, the resulting model is indeed unstable.
This further justifies the need for stability guarantees.

Concluding, the current methods in literature to discover reduced-order nonlinear structural mechanics models from data are typically restricted to quadratic or cubic nonlinearities, do not consider stability constraints or are not formulated in terms of convex optimization problems.
The contribution of this paper is to propose a method to simultaneously address these three shortcomings. Specifically, the method presented here has the following characteristics:

\begin{itemize}
	\item The method infers a reduced-order model (\ref{eqn:model_f}) for dynamics of the form (\ref{eqn:model_f_fom}) with stability guarantees.
	Moreover, it is non-intrusive, i.e., it requires no access to the operators of (\ref{eqn:model_f_fom}).
	\item The method is formulated in terms of a semidefinite program (SDP) with a cost function that penalizes mismatch between the inferred model and the available snapshot data.
	Owing to convexity of the formulation, the inferred operators are globally optimal. 
	Furthermore, the optimization problem can be solved using standard SDP solvers.
	\item The method can be applied to nonlinearities with polynomials of arbitrary degree. 
	To decrease the computational load if the polynomial degree becomes large, a sparse representation of the nonlinearity is proposed.
\end{itemize}

The rest of this paper is organized as follows.
In Section \ref{sct:stabopinf}, we formalize the concept of stability and develop a convex optimization-based approach that infers the reduced operators appearing in (\ref{eqn:model_f}) guaranteeing stability.
To increase computational scalability of the method, in Section \ref{sct:hyperreduction} a hyper-reduction technique is proposed that efficiently selects a sparse selection of monomial terms to represent the nonlinearity.
The method is numerically tested on two examples in Section \ref{sct:examples}.
Finally, conclusions and discussions are given in Section \ref{sct:conclusion}.

Throughout the paper, we use the following notation:
\begin{itemize}
	\item $\norm{ \cdot }$ is the standard Euclidean norm for a vector argument, and the spectral norm for a matrix argument.
	\item $\norm{ s }_{\infty}$ for a signal $s: [0, \infty) \rightarrow \mathbb{R}^{n_s}$ (with $n_s$ the signal dimension) is the supremum norm, i.e., $\norm{s}_\infty := \sup_{t \in [0, \infty)} \norm{s(t)}$.
	\item A $\mathcal{K}$-function is a function $\alpha: [0, \infty) \rightarrow [0, \infty)$ that is continuous, strictly increasing and $\alpha(0) = 0$.
	\item A $\mathcal{K}_\infty$-function $\alpha$ is a $\mathcal{K}$-function with $\lim_{s \rightarrow \infty} \alpha(s) = \infty$. 
	\item A $\mathcal{KL}$-function is a function $\beta: [0, \infty) \times [0, \infty) \rightarrow [0, \infty)$ if $\beta(\cdot, t)$ is a $\mathcal{K}$-function for each $t \geq 0$ with $\lim_{t \rightarrow \infty} \beta(s, t) = 0$ for each $s \geq 0$.
	\item The cardinality of a set $a$ is denoted by $| a |$.
\end{itemize}

%% file: stabopinf.tex
The operators $M$, $C$, $B$ and $f$ of the reduced-order model structure (\ref{eqn:model_f}) should be inferred from a dataset of snapshots of $y(t)$ and $u(t)$.
Note that in experimental settings, it is unreasonable to assume access to snapshots of the full displacement field $y(t)$ (usually, only a limited set of sensor measurements are available).
However, in this paper we focus on the setting where the datasets originate from a high-fidelity simulation model of the system.
In this context, obtaining such snapshots is feasible (whereas obtaining the operators remains infeasible due to software restrictions).

Because learning dynamics directly from data can lead to unstable reduced-order models \cite{martynov2019polynomial}, in this section we develop a convex constraint based on sum-of-squares polynomials which, if satisfied, ensures that the inferred system exhibits certain stability properties.
We start by making precise these stability properties.
In Section~\ref{sct:stab_concepts}, we introduce separate stability properties of the {homogeneous system} (when $u(t) = 0$) and the non-homogeneous system (when $u(t) \neq 0$).
In Section~\ref{sct:comp_model_struct} we propose a model structure that enables the computational verification of these stability properties which is then used to formulate the stabilized operator inference formulation in Section~\ref{sct:opinf_form}.
 
\subsection{Stability concepts}~\label{sct:stab_concepts}
We will work with two main stability properties.
For more background information on these and other notions of stability, consult, e.g., \cite{khalil2002}.
The full state of second-order systems of the form (\ref{eqn:model_f}) is composed of the displacement $x(t)$ and the velocity $\dot{x}(t)$.
The stability theory introduced in this section relies heavily on the full state, for which we introduce the following notation:
\begin{equation}
	z(t) := \begin{bmatrix} x(t)^\mathrm{T} & \dot{x}(t)^\mathrm{T} \end{bmatrix}^\mathrm{T} \in \mathbb{R}^{2r}.
	\label{eqn:z}
\end{equation}

The first notion of stability concerns the boundedness of $z(t)$ in the homogeneous case ($u(t) = 0$).

\begin{dfn}[Global uniform boundedness]
\cite[Definition 4.6]{khalil2002} 
The solutions of (\ref{eqn:model_f}) are globally uniformly bounded if, for $u(t) = 0$, given an arbitrary initial condition $z(0)$ we have that there exists a constant $C$ such that $\norm{ z(t) } \leq C$ for all time $t \geq 0$.
\end{dfn}

We will next introduce the notion of \emph{input-to-state stability} (ISS) which is also applicable to the non-homogeneous case ($u(t) \neq 0$).
\begin{dfn}[Input-to-state stability (ISS)]
\cite[Definition 2.1]{sontag1995characterizations}
The system (\ref{eqn:model_g}) is input-to-state stable if there exists a $\mathcal{KL}$-function $\beta$ and a $\mathcal{K}$-function $\gamma$ such that
\begin{equation}
	\norm{ z(t) } \leq \beta( \norm{ z(0) }, t ) + \gamma( \norm{ u }_\infty ).
\end{equation}
\end{dfn}

Note that ISS implies global uniform boundedness of all solutions (and global asymptotic stability of the origin) of (\ref{eqn:model_f}) if $u(t) = 0$, but the converse does not hold.
Furthermore, as ISS holds for arbitrary inputs (provided $\norm{u}_\infty < \infty$), it is a more powerful and robust property considering that the modeling goal is typically to query the response of the system for many different inputs.
Still, it is useful to consider both stability properties because, as we will see in the sequel, there may be scenarios where imposing ISS is infeasible, whereas global uniform boundedness of the solutions is still feasible.
In particular, this will be the case for systems (\ref{eqn:model_f}) for which $C = 0$ (i.e., in the absence of dissipation).

\begin{remark}
For the homogeneous case ($u(t) = 0$), ISS stability implies global asymptotic stability of the origin ($z(t) = 0$) \cite{sontag1996new}.
\end{remark}

The above stability properties can be verified using (ISS) Lyapunov functions \cite{khalil2002,sontag1995characterizations}, which we will exploit below in the scope of operator inference with stability guarantees.
The following definitions are used frequently in the context of (ISS) Lyapunov analysis:
\begin{dfn}[Radially unbounded function]
A function $\theta: \mathbb{R}^{n_\theta} \rightarrow \mathbb{R}$ is radially unbounded if $\theta(x) \rightarrow \infty$ whenever $\norm{x} \rightarrow \infty$.
\end{dfn}

\begin{dfn}[Positive (semi)definite function]
A function $\theta: \mathbb{R}^{n_\theta} \rightarrow \mathbb{R}$ is positive (semi)definite if $\theta(0) = 0$ and $\theta(x) > 0$ ($\theta(x) \geq 0$) for all $x \neq 0$.
\end{dfn}

\begin{dfn}[Negative (semi)definite function]
A function $\theta: \mathbb{R}^{n_\theta} \rightarrow \mathbb{R}$ is negative (semi)definite if $\theta(0) = 0$ and $\theta(x) < 0$ ($\theta(x) \leq 0$) for all $x \neq 0$.
\end{dfn}

Without any additional assumptions on the nonlinearity $f$ in (\ref{eqn:model_f}), finding a Lyapunov function to guarantee certain stability properties of (\ref{eqn:model_f}) may in general be challenging.
However, we are considering physical systems in this work arising from structural dynamics.
Hence, the principle of Lagrangian mechanics typically provides us with a formulation of $f(x) = \nabla g(x)$ on the basis of a potential as follows (see, e.g., \cite{lurie2013analytical}, \cite[Section 2.1]{ortega1998} and \cite[Section 2.1.1]{sharma2024preserving}):
\begin{equation}
\begin{aligned}
    M \ddot{x} + C \dot{x} + \nabla \left( g(x) \right) = B u,
\end{aligned}
\label{eqn:model_g}
\end{equation}
where $g: \mathbb{R}^{r} \rightarrow \mathbb{R}$ is a nonlinear differentiable function representing the potential energy in the system.
Since this term encompasses potential energy, it is a natural part of Lyapunov functions certifying stability.

Next, we give a possible characterization of a Lyapunov function certifying global uniform boundedness of the solutions of (\ref{eqn:model_f}).

\begin{lemma} \label{lemma:lyapstab}
Solutions of (\ref{eqn:model_f}) are globally uniformly bounded if there exists a continuously differentiable Lyapunov function $V: \mathbb{R}^{2r} \rightarrow [0, \infty)$ such that the following conditions hold:
\begin{enumerate}
	\item[I)] $V$ is radially unbounded.
	\item[II)] $V$ is positive semidefinite.
	\item[III)] $\dot{V}$ is negative semidefinite.
\end{enumerate}
\end{lemma}
\begin{proof}
Pick any $z(0) \in \mathbb{R}^{2r}$.
Consider the set $V_0 := \{ y | V(y) \leq V(z(0)) \}$.
Because $V$ is continuous and radially unbounded, the set $V_0$ is bounded (i.e., there exists $M > 0$ such that $\norm{y} \leq M$ for all $y \in V_0$).
Since $\dot{V}(z(t))$ is non-positive, $z(t) \in V_0$ for all $t \geq 0$ and, hence, $\norm{z(t)}$ is bounded.
\end{proof}
Next, we exploit Lemma~\ref{lemma:lyapstab} to derive conditions on the operators of (\ref{eqn:model_g}) such that the solutions are guaranteed to be globally uniformly bounded (for $u(t) = 0$).

\begin{thm}\label{thm:lyapstab_g}
Solutions of (\ref{eqn:model_g}) with $u(t) = 0$ are globally uniformly bounded if the following conditions are satisfied:
\begin{enumerate}
	\item[I)] $M = M^\mathrm{T} \succ 0$.
	\item[II)] $C = C^\mathrm{T} \succeq 0$.
	\item[III)] $g$ is radially unbounded and positive semidefinite.
\end{enumerate}
\end{thm}
\begin{proof}
Define the following function:
\begin{equation}
	V(x, \dot{x}) = \frac{1}{2} \dot{x}^\mathrm{T} M \dot{x} + g(x).
	\label{eqn:V_FOM_lyapstab}
\end{equation}
$V$ is radially unbounded and positive semidefinite since both terms are independently radially unbounded and positive semidefinite in $x(t)$ and $\dot{x}(t)$, respectively.
Furthermore, the time derivative of $V$ along trajectories of the system is given by
\begin{equation}
\begin{aligned}
	\dot{V}(x, \dot{x}) & = \dot{x}^\mathrm{T} M \ddot{x} + \dot{x}^\mathrm{T} \nabla (g(x) ) \\
		& = \dot{x}^\mathrm{T} \left( M \ddot{x} + \nabla (g(x)) \right) \\
		& = -\dot{x}^\mathrm{T} C \dot{x} \preceq 0.
\end{aligned}
\end{equation}
Using Lemma~\ref{lemma:lyapstab}, we can conclude global uniform boundedness of solutions.
\end{proof}

To analyse ISS, we exploit the concept of an ISS-Lyapunov function.
The next result provides a sufficient condition for ISS of system~\eqref{eqn:model_g}.

\begin{thm}\label{thm:iss_g}
The system (\ref{eqn:model_g}) is ISS if the following conditions are satisfied:
\begin{enumerate}
	\item[I)] $M = M^\mathrm{T} \succ 0$.
	\item[II)] $C = C^\mathrm{T} \succ 0$.
	\item[III)] $g$ is positive definite.
	\item[IV)] $x^\mathrm{T} \nabla( g(x) ) \geq g_0 \norm{x}^2 $ for some $g_0 > 0$ and all $x \in \mathbb{R}^{n}$.
\end{enumerate}
\end{thm}

\begin{proof}
The theorem is a variation of \cite[Theorem 3]{aleksandrov2022iss} with the following differences.

First, we set $2 \eta(\norm{x})$ introduced in \cite[Theorem 3]{aleksandrov2022iss} as follows: $2 \eta(\norm{x}) = g_0 \norm{x}$. This is allowed since $\eta$ should be a $\mathcal{K}_\infty$-function, and $g_0 \norm{x}$ is, indeed, a $\mathcal{K}_\infty$-function.

Second, we allow $M \neq I$.
To accommodate this, we incorporate $M$ into the ISS-Lyapunov function introduced in (6) of \cite{aleksandrov2022iss} as follows:
\begin{equation}
	V = h \left( \dfrac12 \dot{x}^\mathrm{T} \textcolor{red}{M} \dot{x} + g(x) \right) + \dfrac12 x^\mathrm{T} C x + x^\mathrm{T} \textcolor{red}{M} \dot{x},
	\label{eqn:V_iss}
\end{equation}
with $h > 0$.
In~\eqref{eqn:V_iss}, the symbols $h, g(x)$ and $C$ correspond to $h, G(x)$ and $\mathbb{W}$ in (6) of \cite{aleksandrov2022iss}, respectively.
Subsequently, $M$ is also added to the condition on $h$ such that $h \mathbb{W} - I > 0$ in the proof of \cite[Theorem 3]{aleksandrov2022iss}, which becomes $h C - M \succ 0$ in the scope of the current paper.
Note this condition is still feasible because, as in \cite{aleksandrov2022iss}, we have $C = C^\mathrm{T} \succ 0$.
The rest of the proof may be carried out as in \cite{aleksandrov2022iss} using the modified ISS-Lyapunov function~\eqref{eqn:V_iss} and noting that $M = M^\mathrm{T} \succ 0$.
\end{proof}

\subsection{Computationally tractable model structure}~\label{sct:comp_model_struct}
In the previous sections, we have provided a set of sufficient conditions under which the model structure (\ref{eqn:model_g}) has globally uniformly bounded solutions when $u(t) = 0$ (Theorem \ref{thm:lyapstab_g}) or is ISS (Theorem \ref{thm:iss_g}).
In this section, we discuss a computationally tractable way of enforcing these conditions in an operator inference context.
The operator inference problem is usually posed as a convex optimization problem \cite{peherstorfer2016data, qian2020lift}.
Convexity (which arises from the fact that the operators in the dynamics equations are linear) is an attractive property of such problems because any local minimizer is a global minimizer and efficient computational tools exist for solving a wide range of convex problems \cite{boyd2004convex}.
Given these advantages of a convex problem statement, we will also aim for a convex formulation of the operator inference problem for model structures of the form (\ref{eqn:model_g}).
This is challenged by the nonlinearity present in (\ref{eqn:model_g}).
It can immediately be seen that the mass matrix $M$, the damping matrix $C$ and the input-to-state matrix $B$ appear affinely in the dynamics equation (\ref{eqn:model_g}).
Furthermore, the conditions on these operators for stability as given in Theorems \ref{lemma:lyapstab} and \ref{thm:iss_g} are positive (semi)definite constraints.
What remains is a suitable parametrization of the nonlinear term $g(x)$ that 1) is sufficiently general to describe a wide range of nonlinear phenomena; and 2) allows a convex formulation of the operator inference problem.

To achieve this, we will parametrize $g(x)$ as a linear combination of monomials and constrain the resulting linear combination to admit a \emph{sum-of-squares decomposition}. We will explain this in more details by first introducing the definition of a sum-of-squares polynomial.

\begin{dfn}[Sum-of-squares polynomial]
A polynomial $\tau$ is a sum-of-squares polynomial if and only if it can be written in the form
\begin{equation}
	\tau(x) = \sum_{i=1}^{n} \tau_i(x)^2,
\end{equation}
with $\tau_i$, $i = 1, ..., n$, polynomials. The space of all sum-of-squares polynomials (in the variable $x$) is denoted by $\sos$.
\end{dfn}

It is clear that any sum-of-squares polynomial $\tau(x)$ is also nonnegative, but not necessarily positive definite.
We therefore introduce a second definition which includes positive definiteness.

\begin{dfn}[Strictly positive sum-of-squares polynomial]
A polynomial $\tau(x) \in \sos$ is a strictly positive sum-of-squares polynomial if and only if $\tau(x) > 0$ whenever $x \neq 0$. The space of all strictly positive sum-of-squares polynomials (in the variable $x$) is denoted by $\ssos$.
\end{dfn}

Considering the set of sum-of-squares polynomials is attractive from a computational point of view because checking whether a given polynomial is sum-of-squares is equivalent to solving a semidefinite program \cite{parrilo2000structured}.
It is important to note that, in general, the set of nonnegative polynomials is much larger than the set of sum-of-squares polynomials.
However, the good news is that it was shown by \cite{lasserre2007sum} that any nonnegative polynomial $p(x)$ may be approximated arbitrarily closely by a sum-of-squares polynomial $\tau(x)$ (in terms of the $\ell_1$ norm of the coefficients of $\tau(x) - p(x)$).

Motivated by the characterization of sum-of-squares polynomials by semidefinite programs, we assume that the nonlinear function $g$ in (\ref{eqn:model_g}) can be well approximated by a polynomial.
More specifically, we assume that
\begin{equation}
	g(x) \approx k^\mathrm{T} \phi(x),
	\label{eqn:g_k_phi}
\end{equation} 
where the nonlinear mapping $\phi: \R^{n} \rightarrow \R^{n_\phi}$ is a column vector of $n_\phi$ monomials in $x$, i.e.,
\begin{equation}
	\phi(x) = \begin{bmatrix} \prod_{i=1}^{n} x_i^{d_{1,i}} \\ \vdots \\ \prod_{i=1}^{n} x_i^{d_{n_\phi,i}} \end{bmatrix} ,
\end{equation}
where $x_i$, $i = 1, ..., n$, denotes the i-th component of $x$ and $d_{j,i}$, $j = 1, ..., n_\phi$, are non-negative integers.
The column vector $k \in \R^{n_\phi}$ are coefficients of the polynomial $k^\mathrm{T} \phi(x)$ and represent the parametrization of the nonlinear operator that should be inferred.
Substituting (\ref{eqn:g_k_phi}) in (\ref{eqn:model_g}) yields the model structure
\begin{equation}
\begin{aligned}
    M \ddot{x} + C \dot{x} + \nabla \left( k^\mathrm{T} \phi(x) \right) = B u.
\end{aligned}
\label{eqn:model_polynomial}
\end{equation}
The dynamics equation (\ref{eqn:model_polynomial}) is now linear in the operators-to-be-inferred ($M, C, k$ and $B$).
In addition, we relax the positivity constraints of the polynomials $k^\mathrm{T} \phi(x)$ and $x^\mathrm{T} \nabla( k^\mathrm{T} \phi(x) )$ (corresponding to conditions III and IV of Theorem \ref{thm:iss_g}, respectively) by instead finding $k$ such that these polynomials are (strictly) positive sum-of-squares, i.e.:
\begin{equation}
\begin{aligned}
	k^\mathrm{T} \phi(x) & \in \ssos, \\
	x^\mathrm{T} \nabla( k^\mathrm{T} \phi(x)) - \varepsilon \norm{x}^2 & \in \sos,
\end{aligned}
\label{eqn:model_sos_constraints}
\end{equation}
with $\varepsilon$ a small positive constant.

The sum-of-squares constraints (\ref{eqn:model_sos_constraints}) enable the convex parametrization of reduced-order models (ROMs) that exhibit the stability properites introduced in this section.
To summarize, starting from the nonlinear structural dynamics model (\ref{eqn:model_f}), we made the following adaptations:
\begin{enumerate}
	\item First, we let $f(x) = \nabla g(x)$ to arrive at model structure (\ref{eqn:model_g}). This simplifies finding candidate (ISS-)Lyapunov functions for the model (Theorems \ref{thm:lyapstab_g} and \ref{thm:iss_g}), and is motivated by the framework of Lagrangian mechanics.
	\item Second, we assume that the dynamics can be well approximated using $g(x) \approx k^\mathrm{T} \phi(x)$ to arrive at model structure (\ref{eqn:model_polynomial}). This allows us to use a sum-of-squares relaxation to obtain computationally tractable global stability constraints as in (\ref{eqn:model_sos_constraints}).
\end{enumerate}
In the next subsection, the model structure (\ref{eqn:model_g}) developed in this section is combined with data snapshots to arrive at the operator inference formulation.

\subsection{Formulation of the operator inference problem}~\label{sct:opinf_form}
The modeling goal is to find suitable operators of the model structure (\ref{eqn:model_g}) that can accurately reconstruct a simulated or measured dataset and extrapolate to unseen loading profiles.
Note that we consider a non-intrusive setting, which means that we do not have access to the simulation code that generated the dataset, nor do we know the precise details of the governing equations.
It is possible (and given the high complexity of modern FEM codes) that these governing equations cannot be exactly represented using model structure (\ref{eqn:model_g}).
The idea is, however, that model (\ref{eqn:model_g}) are sufficiently general to capture the relevant dynamics.

We assume that snapshots of the displacement coordinates $y(t_i)$, at time instants $t_i$, $i = 1, ..., N$, are available.
These are collected in a snapshot matrix
\begin{equation}
    \mathcal{Y} := \begin{bmatrix} y(t_1) & y(t_2) & \cdots & y(t_N) \end{bmatrix} \in \mathbb{R}^{n \times N}
    \label{eqn:Y}
\end{equation}
along with an accompanying snapshot matrix of the corresponding input:
\begin{equation}
    \mathcal{U} := \begin{bmatrix} u(t_1) & u(t_2) & \cdots & u(t_N) \end{bmatrix} \in \mathbb{R}^{n_u \times N}.
    \label{eqn:U}
\end{equation}
In this work, we apply POD to find the low-dimensional subspace $V$ as in (\ref{eqn:y=Vx}).
Using $V$ and the displacement snapshot matrix $\mathcal{Y}$, we compute the reduced displacement snapshot matrix:
\begin{equation}
	\mathcal{X} = V^\mathrm{T} \mathcal{Y}.
	\label{eqn:X}
\end{equation}

\begin{remark}
It is assumed that the snapshot matrices $\mathcal{Y}$ and $\mathcal{U}$ are collected from a single simulation that covers the expected usage scenario of the model.
If this is not practical, then the derived model reduction methodology presented in this work can be easily extended to accommodate multiple snapshot matrices.
\end{remark}

Similar to (\ref{eqn:X}), we define the snapshot matrices $\dot{\mathcal{X}}$ and $\ddot{\mathcal{X}}$ consisting of the velocity and acceleration, respectively.
Depending on the used simulation software, these quantities may not be available directly.
In that case, one can estimate them using, e.g., finite differencing.

Lastly, we introduce (by a slight abuse of notation) the following shorthand notation:
\begin{equation}
	\nabla \left( k^\mathrm{T} \phi(\mathcal{X}) \right) := \begin{bmatrix} \nabla(k^\mathrm{T} \phi(x(t_1)) \cdots \nabla(k^\mathrm{T} \phi(x(t_N)) \end{bmatrix} \in \mathbb{R}^{r \times N}.
\end{equation}

To quantify the discrepancy between the snapshot data and the model structure, often the residual is used in operator inference \cite{peherstorfer2016data, filanova2023operator}.
The residual at time instance $i$ is given by
\begin{equation}
	M x(t_i) + C \dot{x}(t_i) + \nabla \left( k^\mathrm{T} \phi(x(t_i)) \right) - B u(t_i).
\end{equation} 
The cost function can them be composed by, for example, summing the 2-norm of the residual over all time steps:
\begin{equation}
	\sum_{i=1}^{N} \norm{M x(t_i) + C \dot{x}(t_i) + \nabla \left( k^\mathrm{T} \phi(x(t_i)) \right) - B u(t_i)}_2^2.
\end{equation}
Using the Frobenius norm $\norm{ \cdot }_F$ and the snapshot matrices this can be conveniently expressed as
\begin{equation}
	\norm{ M \ddot{\mathcal{X}} + C \dot{\mathcal{X}} + \nabla (k^\mathrm{T} \phi(\mathcal{X})) - B \mathcal{U} }_F^2.
\end{equation}

Using the snapshot matrices and residual-based cost defined above, we are now ready to formulate the operator inference problem in accordance with Theorem \ref{thm:iss_g} guaranteeing that the model is ISS:
\begin{equation}
\begin{aligned}
	\min_{M,C,k,B} \quad & \norm{ M \ddot{\mathcal{X}} + C \dot{\mathcal{X}} + \nabla (k^\mathrm{T} \phi(\mathcal{X})) - B \mathcal{U} }_F^2, \\
	\textrm{s.t.} \quad & M \succ 0, \\
						& C \succ 0, \\
						& k^\mathrm{T} \phi(x) \in \ssos, \\
						& x^\mathrm{T} \nabla( k^\mathrm{T} \phi(x) ) - \varepsilon \norm{x}^2 \in \sos,
\end{aligned}
\label{eqn:opinfstab_asymp}
\end{equation}
with $\varepsilon > 0$ a small number.

An operator inference formulation guaranteeing models with globally uniformly bounded solutions can be similarly defined by modifying the constraints to be compatible with Theorem \ref{thm:lyapstab_g}:
\begin{equation}
\begin{aligned}
	\min_{M,C,k,B} \quad & \norm{ M \ddot{\mathcal{X}} + C \dot{\mathcal{X}} + \nabla (k^\mathrm{T} \phi(\mathcal{X})) - B \mathcal{U} }_F^2, \\
	\textrm{s.t.} \quad & M \succ 0, \\
						& C \succeq 0, \\
						& k^\mathrm{T} \phi(x) \in \ssos.
\end{aligned}
\label{eqn:opinfstab_lyap}
\end{equation}
The operator inference formulations (\ref{eqn:opinfstab_asymp}) and (\ref{eqn:opinfstab_lyap}) are convex problems.
The cost function promotes finding operators such that the dynamical model accurately matches the data.
Meanwhile, the constraints are data-independent and ensure that only models with appropriate stability properties are considered.

Although the reduction process can drastically reduce the number of state variables, the projection matrix $V$ is a dense matrix.
The effect of this is that sparsity structure that is usually present in the operators of (\ref{eqn:model_f_fom}) is lost.
For example, a FEM semi-discretization of a structural dynamics problem typically leads to a sparse mass matrix $M_y$.
In the reduced coordinate space ($M$) this sparsity is lost.
Although the reduction in state variable is often much higher than the additional matrix entries required to store the dense operators, we need to consider that in our case any sparsity that was present in the nonlinear term $f_y(y(t))$ is also lost.
Consider, for example, the case of structural dynamics under large deformation.
In that setting, the nonlinearity $f_y(y(t))$ becomes a sparse cubic polynomial, with the cross-terms $y_i y_j y_l$ only appearing if the FEM nodes $i, j$ and $l$ are neighbours on the mesh \cite[Section 4]{barbivc2005real}.
Similarly to how the mass matrix becomes dense after reduction, we should expect polynomial nonlinearities to become dense.
That is, all possible monomials of the state variables will have a non-zero contribution.

In the next section, we will show that the number of monomials grows rapidly as a function of $r$.
In fact, it grows much more rapidly than the number of matrix entries of the dense operators $M$ and $C$ of (\ref{eqn:model_polynomial}).
Because the number of monomials directly affects the computational complexity of the operator inference problem (\ref{eqn:opinfstab_asymp}), controlling the growth becomes vital to ensure computational scalability of the method.
Hence, the hyperreduction scheme developed in the next section is an important ingredient to achieve this.

%% file: hyperreduction.tex
\subsection{Sparsity for sum-of-squares polynomials}
The number of unique monomials of degree at most $d$ in $r$ variables is given by the binomial coefficient
\begin{equation}
	n_{\phi} = {r + d \choose d}.
\end{equation}
Without loss of generality we can assume that $d$ is even.
Namely, if $d$ were odd, then the highest power term of $k^\mathrm{T} \phi$ would be odd.
However, the corresponding coefficient in $k$ must necessarily be 0 to ensure non-negativity of $k^\mathrm{T} \phi$.
Due to similar reasoning, there can be no linear monomials with non-zero coefficients in the polynomial.
Hence, we can exclude these $r$ terms.
Finally, we may assume that $k^\mathrm{T} \phi(0) = 0$, i.e., the constant monomial is 0.
Let $n_{\phi,\mathrm{full}}$ be the number of monomials in $k^\mathrm{T} \phi(x)$.
Then, after exclusion of the linear monomials and the constant term, we have
\begin{equation}
	n_{\phi,\mathrm{full}} = {r + d \choose d} - r - 1.
\end{equation}
The number of monomials grows quickly as either $r$ or $d$ are increased.
For example, if $d$ is fixed then $n_{\phi,\mathrm{full}} = \mathcal{O}(r^{d-1})$ grows polynomially in $r$.
A large number of monomials slows down both the offline phase (solving the operator inference problem) and the online phase (simulating the ROM).
To combat the growth of the number of monomials, one possibility is to reduce the number of monomials through a sparse selection procedure.
An important aspect is that, due to the sum-of-squares constraints imposed on $k^\mathrm{T} \phi$ and $\nabla ( k^\mathrm{T} \phi )$, the monomials cannot be selected independently.
To illustrate this point, consider the following set of monomials:
\begin{equation}
	\phi_{\mathrm{full}}(x) := \begin{bmatrix} x_1^2 & x_1^3 & x_1^4 \end{bmatrix}^{\mathrm{T}}.
\end{equation}
For the above example, we could discard $x_1^2$ in our sparse selection procedure to arrive at $\phi$ and then determine coefficients $k_1, k_2$ such that $k^\mathrm{T} \phi(x) = k_1 x_1^3 + k_2 x_1^4 \in \sos$.
However, every possible sum-of-squares decomposition requires $k_1 = 0$, because the monomial of lowest degree in a sum-of-squares decomposition must have even degree.
Using a similar argument, discarding $x_1^4$ from $\phi_\mathrm{full}$ will also require us to discard $x_1^3$.

To obtain a sparse selection of monomials that is consistent with a sum-of-squares formulation, we propose instead to work with clusters of the variables that appear jointly in the sum-of-squares decomposition.
That is, we generate all monomials that could appear if we restrict the polynomial $k^\mathrm{T} \phi$ to be a sum-of-squares involving $n_\chi$ polynomials in at most $\Theta \leq r$ variables:
\begin{equation}
k^\mathrm{T} \phi(x) = \sum_{i=1}^{n_\chi} \chi_i( x_{m_{i,1}}, ..., x_{m_{i,\Theta}} )^2,
\label{eqn:kphi_decomp}
\end{equation}
with $\chi_i$, $i = 1, ..., n_\chi$, polynomials (of degree $d/2$) and $m_{i,j} \in \{1, ..., r\}$, $j = 1, ..., \Theta$.
When $\Theta = 2$, then the clustering of variables can be visualized in a graph structure with $r$ vertices (representing the variables) connected by $n_\chi$ edges (representing the $\chi_i$ polynomials). See Figure \ref{fig:sparse_selection}.
For $\Theta \geq 3$, the clustering is related to a hypergraph, with each $\chi_i$ polynomial associated with a $\Theta$-hyperedge.

\begin{figure}
	\centering
	\begin{tikzpicture}
		\node [circle, fill=black!20] (x1) at (0, 0) {$x_1$};
		\node [circle, fill=black!20] (x2) at (3, 0) {$x_2$};
		\node [circle, fill=black!20] (x3) at (0, 3) {$x_3$};
		\node [circle, fill=black!20] (x4) at (3, 3) {$x_4$};
	    \draw (x1) -- node[above] {$\chi_1(x_1, x_2)$} (x2);
	    \draw (x1) -- node[left] {$\chi_2(x_1, x_3)$} (x3);
	    \draw (x3) -- node[above] {$\chi_3(x_3, x_4)$} (x4);
	\end{tikzpicture}
	\caption{The sparse selection of monomials with $\Theta = 2$ can be visualized by an undirected graph. Each edge indicates a clustering of two variables and generates a set of monomials $\chi_i$.}
	\label{fig:sparse_selection}
\end{figure}
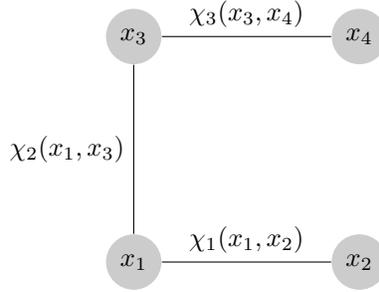

Assuming that for each $\chi_i$ we generate all monomials of degree at most $d$, the number of monomials in decomposition (\ref{eqn:kphi_decomp}) is upper bounded by
\begin{equation}
	n_{\phi,\mathrm{sparse}} := n_\chi \left( {\Theta + d \choose d} - \Theta - 1 \right).
\end{equation}
Note that this is a strict upper bound if any variable $x_i$ belongs to more than 1 cluster, because the powers of $x_i$ will be reduplicated.

If we choose $n_\chi$ and $\Theta$ such that $n_{\phi,\mathrm{sparse}} < n_{\phi,\mathrm{full}}$, then a reduction in the number of monomials of $k^\mathrm{T} \phi(x)$ is achieved.
As a consequence, the computational cost of the online and offline phase is reduced.
However, a careful selection of the clusters should be made to minimize loss in accuracy.
In the next section, we will propose an algorithm to make this selection based on the snapshot data.

\subsection{Clustering algorithm}
A minimum requirement for any clustering is that each variable $x_i$, $i = 1, ..., r$ appears in a polynomial $\chi_j$, $j = 1, ..., n_\chi$ at least once.
In terms of the hypergraph analogy, this means that each vertex should be included in at least 1 hyperedge.
Because each cluster contains $\Theta$ variables, the minimum number of clusters to cover all variables is $n_\chi \geq r / \Theta$.
However, selecting the set of clusters leading to the most accurate ROM without an exhaustive search is far from trivial.
Assuming that POD is used for the dimensionality reduction, we propose an iterative greedy selection algorithm based on ordering the hyperedges by the \emph{relative dominance importance value} as introduced in \cite{martynov2019polynomial} for quadratic and cubic monomials.
This relative dominance importance value is computed for each hyperedge as follows.
First, we associate with each hyperedge the monomial of the $\Theta$ vertices belonging to that hyperedge, e.g., with hyperedge $(x_1, x_2, x_3)$ we associate the monomial $x_1 x_2 x_3$.
Then, we compute for each monomial the relative dominance importance value given by
\begin{equation}
	s_{i_1, ..., i_\Theta} := \prod_{j=1}^{\Theta} { I_j },
	\label{eqn:importance_index}
\end{equation}
where
\begin{equation}
	I_i := \frac{ \sigma_i^2 }{ \sum_{j=1}^{r} \sigma_j^2 }
\end{equation}
and $\sigma_1 \geq \sigma_2 \geq \ldots \geq \sigma_r \geq 0$ are the singular values of the POD basis vectors in $V$.

The relative dominance importance values $s_{i_1, ..., i_\Theta}$ give a heuristic ordering of the importance of all possible clusters with respect to the data.
This ordering will be used to perform the greedy selection algorithm which is detailed in Algorithm \ref{alg:sparse_selection}.
We greedily select clusters from the importance values in two stages.
In the first stage, we ensure that each state variable is selected at least once.
We do this by keeping track of a list of state variables that are not yet part of a selected cluster.
We update this list by, at each iteration, considering only the clusters that have at least one state variable from the list and selecting the cluster with highest importance index.
In the second stage, we select freely the remaining clusters in order of importance until a pre-specified maximum number of monomials has been reached.

\begin{algorithm}
	\caption{Sparse selection of monomials using clustering.}
	\label{alg:sparse_selection}
	\begin{algorithmic}
		\Require Snapshot matrix $\mathcal{X} \in \mathbb{R}^{r \times N}$, maximum number of monomials $n_\phi$, cluster size $\Theta$, maximum monomial degree $d$.
		\Ensure Selected set of monomials $\phi$.
		\State Compute the importance values $s_{i_1, ..., i_\Theta}$ according to (\ref{eqn:importance_index}).		
		\State Initialize selected set of monomials $\phi \leftarrow \{ \}$ and clusters $\mathcal{I} \leftarrow \{ \}$.
		\State \comment{\# Stage 1: greedily select clusters such that each state variable is included at least once}
		\State $x_{\mathrm{open}} \leftarrow \{ 1, \cdots, r \}$.
		\While{$x_{\mathrm{open}} \neq \{ \}$}
			\State Select highest $s_{i_1, ..., i_\Theta}$ with $i_j \in x_{\mathrm{open}}$ for some $j \in \{1, ..., \Theta\}$.
			\State $\mathcal{I} \leftarrow \mathcal{I} \cup \{ \{ i_1, ..., i_\Theta \} \}$.
			\State Generate $\phi_{i_1, ..., i_\Theta}$, the set of all monomials of at most degree $d$ containing variables $x_{i_1}, ..., x_{i_\Theta}$.
			\State $\phi \leftarrow \phi \cup \phi_{i_1, ..., i_\Theta}$.
			\State $x_\mathrm{open} \leftarrow x_\mathrm{open} \setminus \{i_1, ..., i_\Theta \}$.
		\EndWhile
		\State \comment{\# Stage 2: greedily select remaining clusters}
		\While{$| \phi | < n_\phi$}
			\State Select highest $s_{i_1, ..., i_\Theta}$ with $\{i_1, ..., i_\Theta\} \notin \mathcal{I}$.
			\State $\mathcal{I} \leftarrow \mathcal{I} \cup \{ \{ i_1, ..., i_\Theta \} \}$.
			\State Generate $\phi_{i_1, ..., i_\Theta}$, the set of all monomials of at most degree $d$ containing variables $x_{i_1}, ..., x_{i_\Theta}$.
			\State $\phi \leftarrow \phi \cup \phi_{i_1, ..., i_\Theta}$.
		\EndWhile
		
	\end{algorithmic}
\end{algorithm}

%% file: examples.tex
In this section, we evaluate the performance of the proposed method in terms of accuracy and computational complexity.
We present two examples.
In Section \ref{sct:corner_brace}, we consider a 2D FEM model of a corner brace and in Section~\ref{sct:pistonrod}, we consider a 3D FEM model of a piston rod. For the numerical experiments, the following set-up is used:
\begin{itemize}
	\item All software is run on a computer with an Intel i7-9750H CPU and 32 GB of RAM.
	\item The simulation of the FEM models is done using the package \textsc{dolfinx} (version 0.7.0) in Python 3.10.12.
	\item The operator inference problem is solved using the MOSEK solver (version 10.1.20) and \textsc{YALMIP} (version 20230622) in MATLAB R2024a.
\end{itemize}

\subsection{Example 1: corner brace} \label{sct:corner_brace}
We investigate the reduced-order deformation modelling of a rubber corner brace.
From the many computational models for rubbers available in literature (see, e.g., \cite{ali2010review}), we select the Neo-Hookean model.
The motivation for this choice is that it matches the numerical example of \cite{martynov2019polynomial}, providing a better comparison.
The geometry and discretization of the spatial domain can be seen in Figure \ref{fig:cornerbrace_mesh}.
The total number of nodes in the mesh is 577.
A homogeneous Dirichlet boundary condition is applied to the 8 nodes on the bottom horizontal edge of the brace ($y = 0$).
Hence, the 569 non-Dirichlet nodes lead to a full-order state-space representation of size 1,138.
For the mechanical loading of the model, we impose a vertical load on the upper-left vertical edge of the model.
The magnitude of this load is time-varying and we select 2 profiles which we use for inference and validation, respectively:
\begin{equation}
	u_{\mathrm{inf}}(t) = 4 \sin ( 0.2 \pi t),
	\label{eqn:cornerbrace_uinf}
\end{equation}
and
\begin{equation}
	u_{\mathrm{val}}(t) = 2.5 \sin ( (0.1 + 0.1 \cos(t)) t ).
	\label{eqn:cornerbrace_uval}
\end{equation}
In Figure~\ref{fig:cornerbrace_deformed} we visualize the Dirichlet boundary condition and loading term on the corner brace model in a deformed state.
\begin{figure}[!htb]
	\centering
	\begin{subfigure}[b]{0.45\textwidth}
		\centering
		\includegraphics[width=\textwidth]{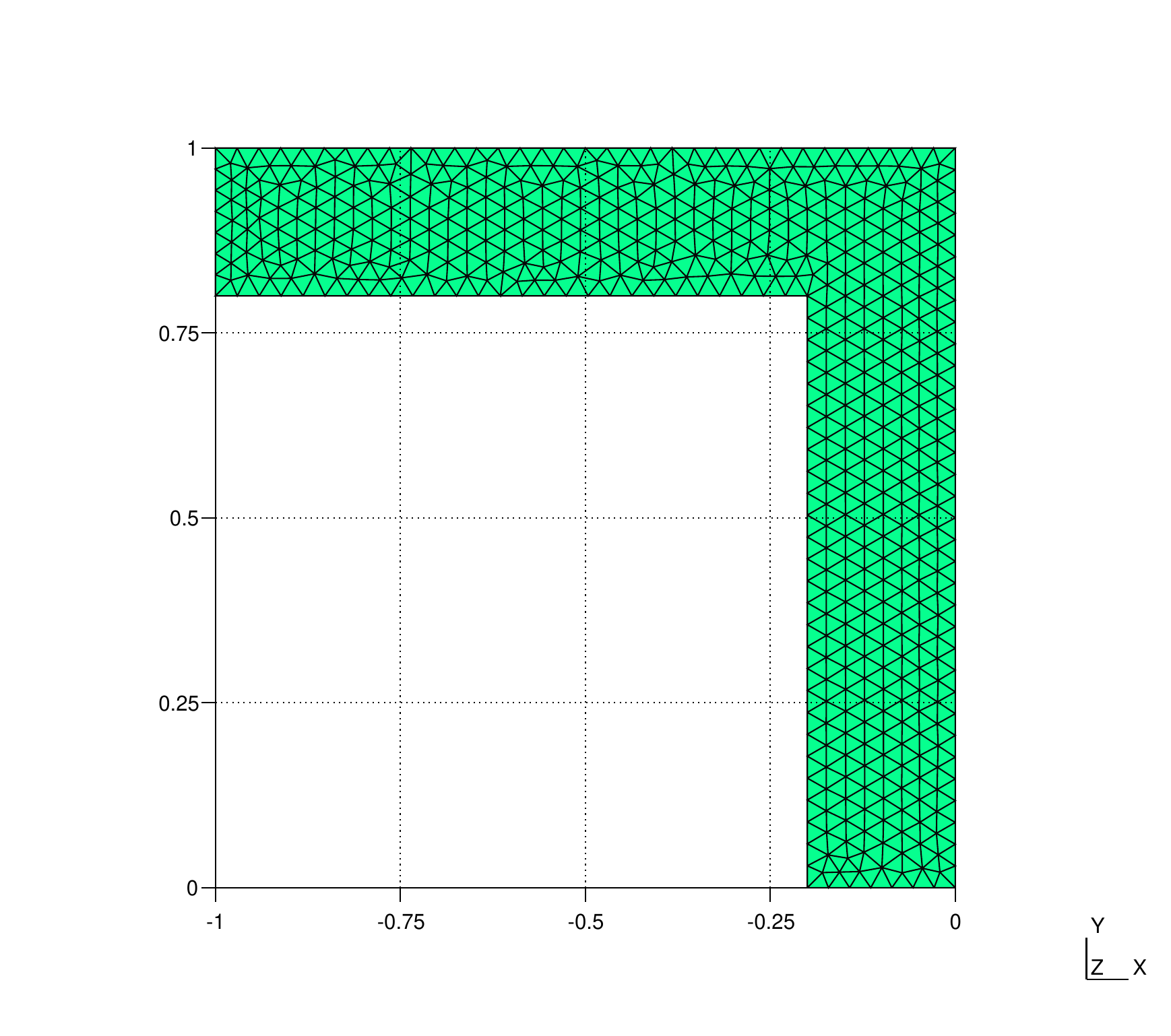}
		\caption{Geometry and mesh of the undeformed corner brace model.}
		\label{fig:cornerbrace_mesh}
	\end{subfigure}
	\hfill
	\begin{subfigure}[b]{0.45\textwidth}
		\centering
		\includegraphics[width=0.5\textwidth]{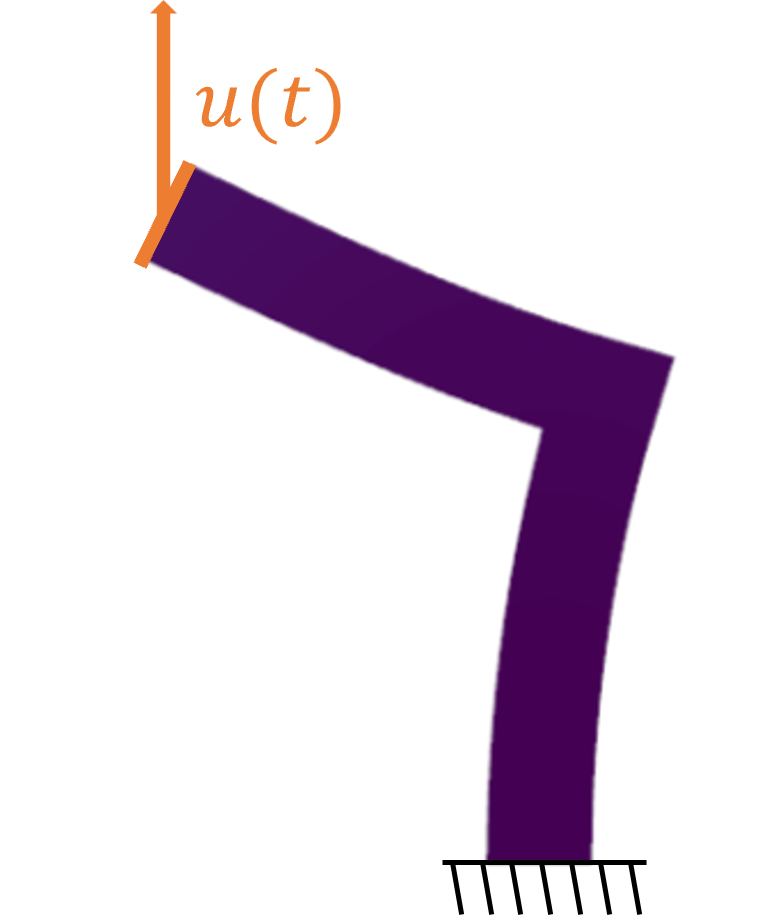}
		\caption{Deformed corner brace model with visualization of the forcing term $u(t)$ and Dirichlet boundary condition.}
		\label{fig:cornerbrace_deformed}
	\end{subfigure}
	\caption{Geometry, mesh and boundary conditions of the corner brace example.}
	\label{fig:cornerbrace}
\end{figure}
Each of the 2 simulations are performed in the time interval $[0, 20]$ seconds.
Within this interval we sample 200 snapshots of the state at regular intervals.
Based on the snapshot dataset obtained with the inference profile (\ref{eqn:cornerbrace_uinf}), we construct a reduced-order model using the following sets of parameters:
\begin{enumerate}
	\item Using POD, we determine a reduced state basis $V$ of dimension $r \in \{2, 3, 7\}$ and construct the reduced state snapshot matrix (\ref{eqn:X}).
	\item For the monomial set $\phi$, we use the following parameter sets:
	\begin{enumerate}
		\item Monomial degree $d \in \{ 2, 4 \}$. Note that $d = 2$ corresponds to linear ROMs.
		\item For the ROMs with degree $d = 4$ and dimension $r > 2$ we use Algorithm \ref{alg:sparse_selection} to determine a sparse selection of monomials. We use $n_\chi = r - 1$ and $\Theta = 2$.
	\end{enumerate}
	\item For each combination of $r$ and the monomial set $\phi$, we solve the optimization problem (\ref{eqn:opinfstab_lyap}) to find the operators $M$, $C$, $B$, and $k$.
\end{enumerate}
The above procedure leads to 8 ROMs in total.
We compare the ROMs based on the reconstruction error of the full state.
That is, we first simulate the inferred models using the validation input (\ref{eqn:cornerbrace_uval}).
Then, we compute the error of the models on this validation dataset as the normalized mean of the reconstruction error evaluated at $N = 200$ equally spaced time samples $t_i$, i.e.:
\begin{equation}
	\mathrm{err}_\mathrm{val} := \frac{ \sum_{i=1}^{N} \norm{ y(t_i) - V \tilde{x}(t_i) }_2 }{ \sum_{i=1}^{N} \norm{ y(t_i) }_2 },
	\label{eqn:accuracy}
\end{equation}
where $\tilde{x}$ denotes the state of the inferred models.
We also compute in this way the error for the inference dataset ($\mathrm{err}_\mathrm{inf}$) analogous to (\ref{eqn:accuracy}):
\begin{equation}
	\mathrm{err}_\mathrm{inf} := \frac{ \sum_{i=1}^{N} \norm{ y(t_i) - V \tilde{x}(t_i) }_2 }{ \sum_{i=1}^{N} \norm{ y(t_i) }_2 }.
	\label{eqn:accuracy_inf}
\end{equation}

For each run, we also measure computational time of two main procedures.
First, we compute the time $t_\mathrm{inf}$ used to solve the operator inference problem (\ref{eqn:opinfstab_lyap}). For ROMs for which Algorithm \ref{alg:sparse_selection} is used, its running time is included in $t_\mathrm{inf}$ as well.
Second, the time $t_\mathrm{sim}$ used to simulate the ROM on the validation dataset.

The results are reported in Table \ref{tbl:cornerbrace}.
By inspecting the validation error ($\mathrm{err}_{\mathrm{val}}$), it becomes obvious that all linear ROMs fail to capture the nonlinear behaviour of the model.
Increasing the number of reduced basis vectors has no effect on the error.
This confirms the need for a nonlinear ROM for this example.
The nonlinear ROMs achieve similar error on the inference dataset and display good generalization capabilities on the validation dataset.
In fact, the nonlinear ROMs with $r = 3$ achieve similar validation error ($\mathrm{err}_\mathrm{val}$) to the nonlinear ROMs with $r = 7$.
This indicates that the dynamics contained in the snapshot data is already well-represented by a 3-dimensional basis.
If we increase $r$, the validation error is significantly higher if we do not use a sparse selection of monomials, indicating the usefulness of the proposed sparse selection strategy.
When $r = 7$, the fast growth in the number of monomials of degree at most 4 (322) becomes apparent.
This results in a significant difference in computational time (both inference and simulation) between the sparse and non-sparse ROMs.

\begin{table}
	\centering
	\begin{tabular}{|c|c|c|c|r|r|r|r|} \hline
	$r$ & $d$ 	& $\Theta$ 	& $n_\phi$ 	& $\mathrm{err}_\mathrm{inf}$ 	& $\mathrm{err}_\mathrm{val}$ 	& $t_\mathrm{inf}$ (s) & $t_\mathrm{sim}$ (s) \\ \hline
	2 	& 2 		& - 			& 3 			& \num{3.01e-1} 					& \num{3.30e-1} 					& \num{1.68e0} 	& \num{3.11e-2} \\
	3 	& 2 		& - 			& 6 			& \num{3.04e-1} 					& \num{3.45e-1} 					& 0.95 			& \num{2.87e-2} \\
	7 	& 2 		& - 			& 28 		& \num{2.94e-1}					& \num{3.45e-1} 					& \num{1.56e0} 	& \num{4.84e-2} \\
	2 	& 4 		& - 			& 12 		& \num{2.14e-2} 					& \num{6.54e-2} 					& \num{1.06e0} 	& \num{7.21e-2} \\
	3 	& 4 		& - 			& 31 		& \num{1.05e-2} 					& \num{4.73e-2} 					& \num{1.02e0} 	& \num{9.41e-2} \\
	3 	& 4 		& 2 			& 21 		& \num{1.69e-2} 					& \num{5.47e-2} 					& \num{1.12e0} 	& \num{6.89e-2} \\
	7 	& 4 		& - 			& 322 		& \num{6.10e-2} 					& \num{4.16e-2} 					& \num{7.54e0} 	& \num{2.57e0} \\
	7 	& 4 		& 2 			& 57 		& \num{2.31e-2} 					& \num{6.05e-2} 					& 2.12 			& \num{2.22e-1} \\ \hline
	\end{tabular}
	\caption{Results for the corner brace numerical example. A ``-" for $\Theta$ indicates no sparse selection of monomials was used.}
	\label{tbl:cornerbrace}
\end{table}

\subsubsection{Reduced-data setting}
Because operator inference is a data-driven technique, the accuracy of the ROMs depends on the amount of data that is supplied in the inference stage.
To investigate this dependence, we repeated the inference of nonlinear ROMs for $r = 3$ but discarded the last 105 samples of the inference dataset.
We compare three ROMs: 1) without stability constraints; 2) with stability constraints; and 3) with stability constraints and sparsity.
For each ROM, we compute the error on the validation dataset and we plot the x-displacement over time of the node on the top-left corner.

The results are shown in Figure~\ref{fig:cornerbrace_reduced_data}.
The ROM without stability constraints (Figure~\ref{fig:cornerbrace_reduced_data_nonstab}) exhibits strong oscillations after approximately 9 seconds.
As a result, the associated validation error is high (15).
Adding stability constraints (Figure~\ref{fig:cornerbrace_reduced_data_stab}) shows a significant decrease in error to 0.13.
This highlights the regularizing function of the stability constraints: in the reduced-data regime these constraints are vital to inferring ROMs that behave physically accurate.
The error is decreased further to 0.10 after employing the clustering-based sparsity approach (Figure~\ref{fig:cornerbrace_reduced_data_stab_sparse}).

\begin{figure}
	\centering
	\begin{subfigure}[b]{0.45\textwidth}
		\centering
		\includegraphics[width=\textwidth]{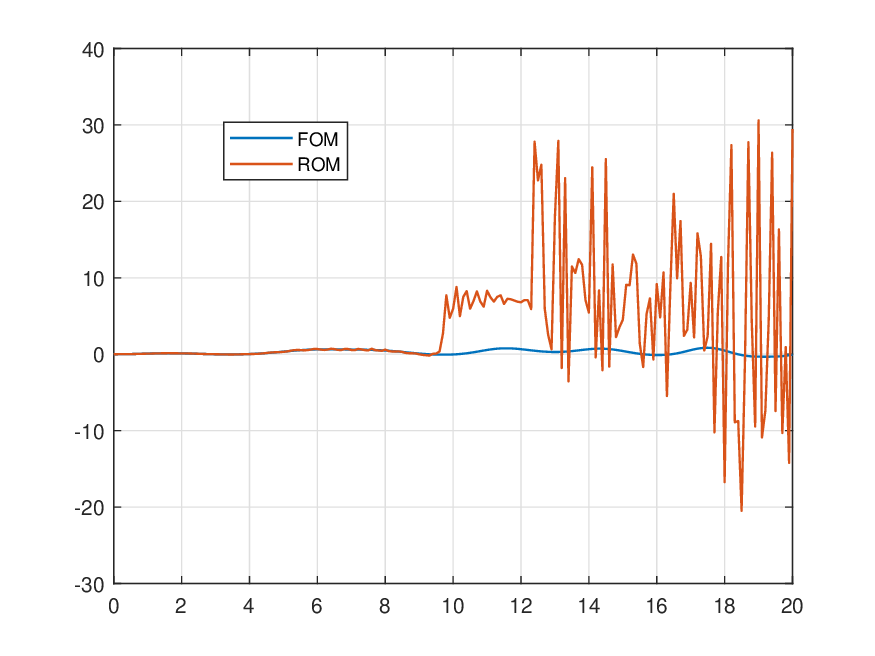}
		\caption{Without stability constraints ($\mathrm{err}_\mathrm{val} = 15$).}
		\label{fig:cornerbrace_reduced_data_nonstab}
	\end{subfigure}
	\begin{subfigure}[b]{0.45\textwidth}
		\centering
		\includegraphics[width=\textwidth]{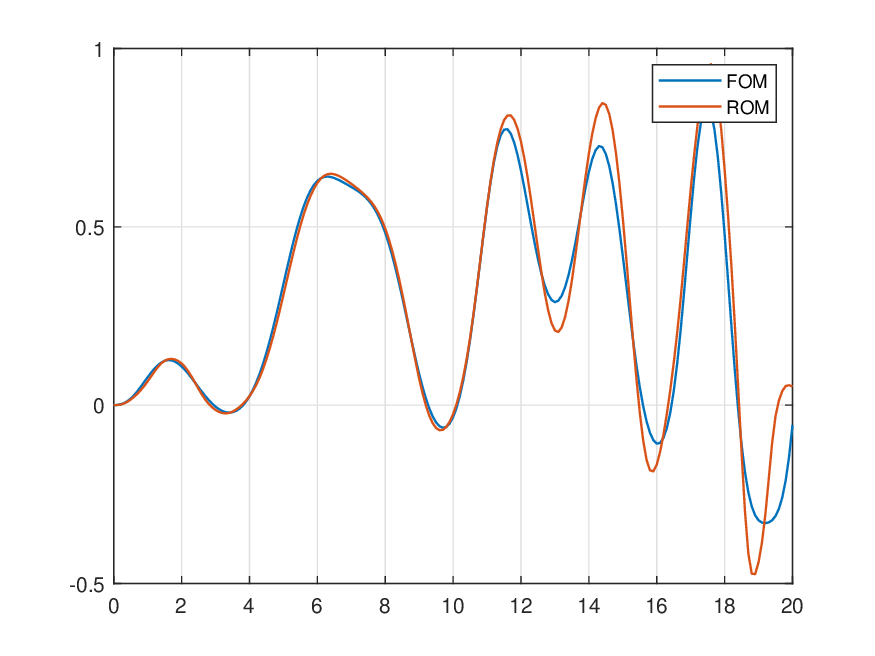}
		\caption{With stability constraints ($\mathrm{err}_\mathrm{val} = 0.13$).}
		\label{fig:cornerbrace_reduced_data_stab}
	\end{subfigure}
	\begin{subfigure}[b]{0.45\textwidth}
		\centering
		\includegraphics[width=\textwidth]{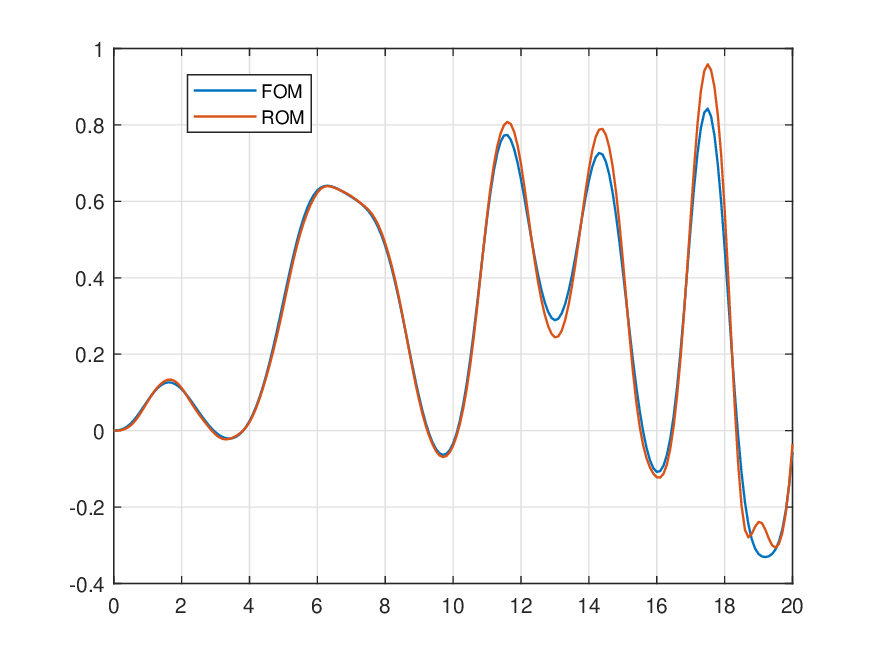}
		\caption{With stability constraints and sparsity ($\mathrm{err}_\mathrm{val} = 0.10$).}
		\label{fig:cornerbrace_reduced_data_stab_sparse}
	\end{subfigure}
	\caption{ROM performance in the reduced data setting.}
	\label{fig:cornerbrace_reduced_data}
\end{figure}

\subsection{Example 2: piston rod}\label{sct:pistonrod}

We consider a 3D model of a steel piston rod as shown in Figure \ref{fig:pistonrod_geometry}.
The example is based on \cite[Section 6.2]{martynov2019polynomial}.
We also use the same material parameters: Young's modulus $E = \SI{210}{\giga\pascal}$, Poisson's ratio $\nu = 0.3$ and mass density $\rho = \SI{7860}{\kg\per\m\cubed}$.
After discretization, the model has 4686 degrees-of-freedom.
The inner surface of the hole on the right-hand side is fixed with a Dirichlet boundary condition.
The inner surface of the hole on the left-hand side is loaded with a surface traction $u(t)$ which is everywhere parallel with the z-axis, i.e., perpendicular to the plane of the rod.
Torsion is generated by dividing the inner surface of the hole on the left-hand side into an upper and a lower part and prescribing a surface traction on each part with opposite sign.
Namely, the surface traction on the upper half of the inner surface is directed upward, while the surface traction on the lower half is directed downward.
The magnitude of the load for inference and validation is given by:
\begin{equation}
	u_{\mathrm{inf}}(t) = 4.0 \sin ( 200 \pi t) \, \si{\giga\pascal},
	\label{eqn:piston_rod_uinf}
\end{equation}
and
\begin{equation}
	u_{\mathrm{val}}(t) = 4.0 \sin ( 300 \pi t ) \, \si{\giga\pascal}.
	\label{eqn:piston_rod_uval}
\end{equation}
The simulations are performed in the time interval $t \in [0, 0.020]$ seconds.

\begin{figure}
	\centering
	\begin{subfigure}[t]{0.45\textwidth}
		\includegraphics[width=\textwidth]{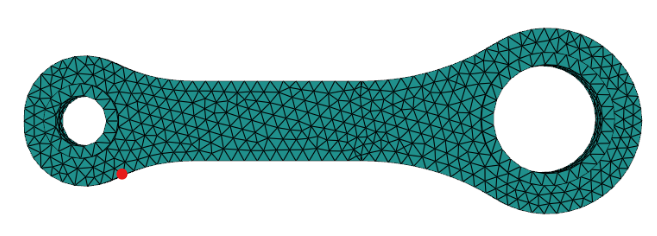}
		\caption{Undeformed. The red dot marks the degree-of-freedom for which time-domain plots are presented in Figure~\ref{fig:pistonrod_results}.}
		\label{fig:pistonrod_undeformed}
	\end{subfigure}
	\begin{subfigure}[t]{0.45\textwidth}
		\includegraphics[width=\textwidth]{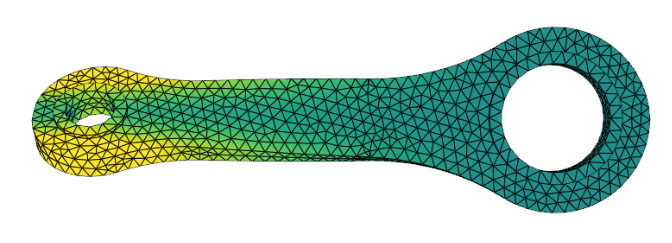}
		\caption{Deformed.}
		\label{fig:pistonrod_deformed}
	\end{subfigure}
	\caption{Geometry and mesh of the piston rod model.}
	\label{fig:pistonrod_geometry}
\end{figure}

The procedure to compute ROMs for this example is similar to the procedure for the corner brace model described in Section~\ref{sct:corner_brace}.
The only difference is in the parameter ranges: for the order of the ROMs we use $r \in \{3, 4, 5\}$, and for the sparse selection (Algorithm~\ref{alg:sparse_selection}) we use $\Theta \in \{2, 3\}$.

The results are reported in Table~\ref{tbl:pistonrod}.
Note that the combination $r = 3, \Theta = 3$ corresponds to the scenario that no sparse selection is used (indicated by $\Theta = $ ``$-$" in the table).
As for the corner brace model, none of the linear ROMs ($d = 2$ in Table~\ref{tbl:pistonrod}) capture the nonlinear dynamics well.
The nonlinear ROMs ($d = 4$ in Table~\ref{tbl:pistonrod}) all achieve inference error below $\num{0.08}$ and validation error below $\num{0.1}$.
The lowest validation error is achieved for $r = 4$ without sparsity.
Increasing $r$ to $5$ decreases inference error from $\num{3.22e-2}$ to $\num{1.28e-2}$ but the validation error is increased, indicating that overfitting occurs at this point.

In Figure~\ref{fig:pistonrod_results} we show time-domain plots of the x-displacement of the point-of-interest of the piston rod (marked by the red dot in Figure~\ref{fig:pistonrod_undeformed}).
For this particular degree-of-freedom, it can be observed that the ROM containing all polynomial terms ($\Theta = 4$) more closely follows the FOM than the 2 ROMs with sparsification applied ($\Theta = 3$ and $\Theta = 2$).
This is in line with the validation error associated with these ROMs given in Table~\ref{tbl:pistonrod}.
The sparsification further reduces the size (and simulation time) of the ROMs but at the expense of a drop in accuracy.

\begin{table}
	\centering
	\begin{tabular}{|c|c|c|c|r|r|r|r|} \hline
	$r$ & $d$ 	& $\Theta$ 	& $n_\phi$ 	& $\mathrm{err}_\mathrm{inf}$ 	& $\mathrm{err}_\mathrm{val}$ 	& $t_\mathrm{inf}$ (s) & $t_\mathrm{sim}$ (s) \\ \hline
	3 	& 2 		& - 			& 6 			& \num{4.13e-1} 					& \num{2.59e-1} 					& \textbf{1.00} & \textbf{\num{2.39e-2}} \\
	4 	& 2 		& - 			& 10 		& \num{3.83e-1} 					& \num{2.35e-1} 					& 1.23 	& \num{2.45e-2} \\
	5 	& 2 		& - 			& 15 		& \num{3.78e-1}					& \num{2.33e-1} 					& 1.33 	& \num{4.23e-2} \\
	3 	& 4 		& - 			& 31 		& \num{1.08e-1} 					& \num{7.92e-2} 					& 1.33 	& \num{5.31e-2} \\
	3 	& 4 		& 2 			& 21 		& \num{7.00e-2} 					& \num{9.19e-2} 					& 1.36 	& \num{3.72e-2} \\
	4 	& 4 		& - 			& 65 		& \num{3.22e-2} 					& \textbf{\num{4.03e-2}} 		& 2.27 	& \num{6.30e-2} \\
	4	& 4		& 3			& 50			& \num{3.63e-2}					& \num{5.18e-2}					& 3.17  & \num{4.92e-2} \\
	4 	& 4 		& 2 			& 30 		& \num{7.07e-2} 					& \num{1.03e-1} 					& 2.05 	& \num{4.30e-2} \\
	5 	& 4 		& - 			& 120 		& \textbf{\num{1.28e-2}}			& \num{7.24e-2} 					& 4.32 	& \num{3.58e-1} \\
	5 	& 4 		& 3 			& 69 		& \num{2.65e-2} 					& \num{4.70e-2} 					& 4.22 	& \num{8.31e-2} \\ 
	5 	& 4 		& 2 			& 39 		& \num{7.18e-2} 					& \num{9.14e-2} 					& 2.20 	& \num{7.14e-2} \\ 
	\hline
	\end{tabular}
	\caption{Results for the piston rod numerical example. A ``-" for $\Theta$ indicates no sparse selection of monomials was used.}
	\label{tbl:pistonrod}
\end{table}

\begin{figure}
	\centering
	\begin{subfigure}[t]{0.45\textwidth}
		\includegraphics[width=\textwidth]{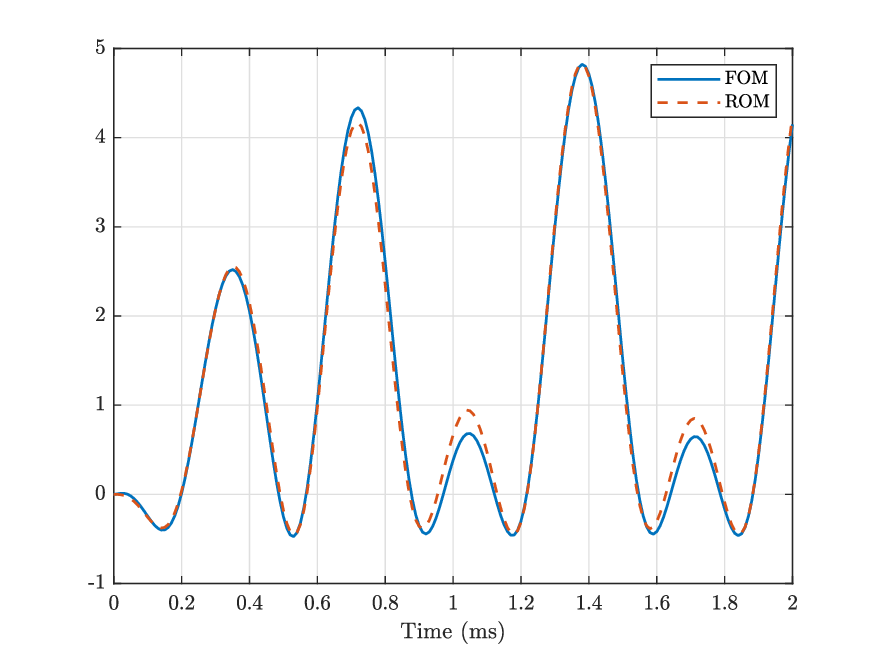}
		\caption{$\Theta = 4$, i.e., no sparsity ($n_\phi = 65$).}
	\end{subfigure}
	\begin{subfigure}[t]{0.45\textwidth}
		\includegraphics[width=\textwidth]{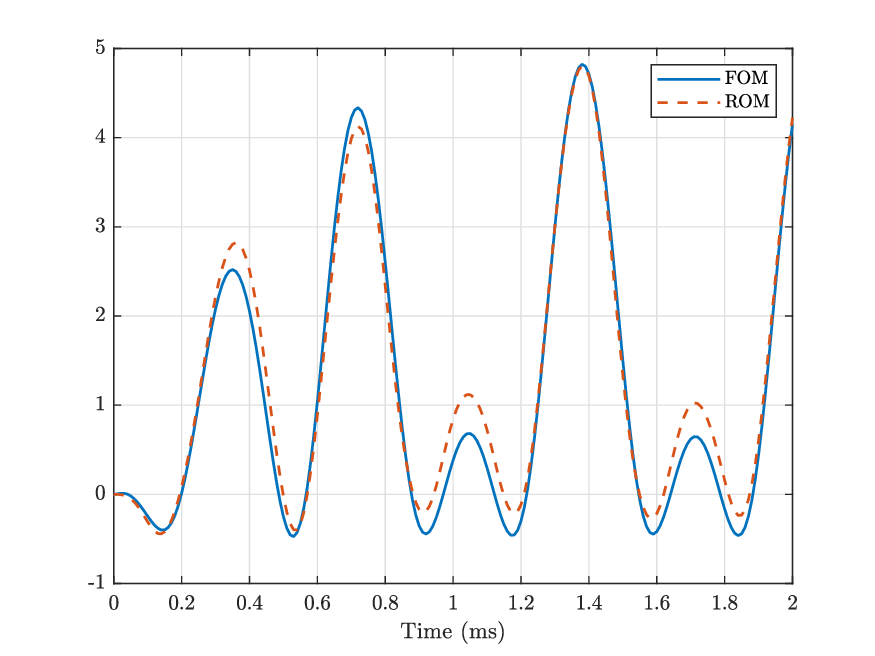}
		\caption{$\Theta = 3$ ($n_\phi = 50$).}
	\end{subfigure}
	\begin{subfigure}[t]{0.45\textwidth}
		\includegraphics[width=\textwidth]{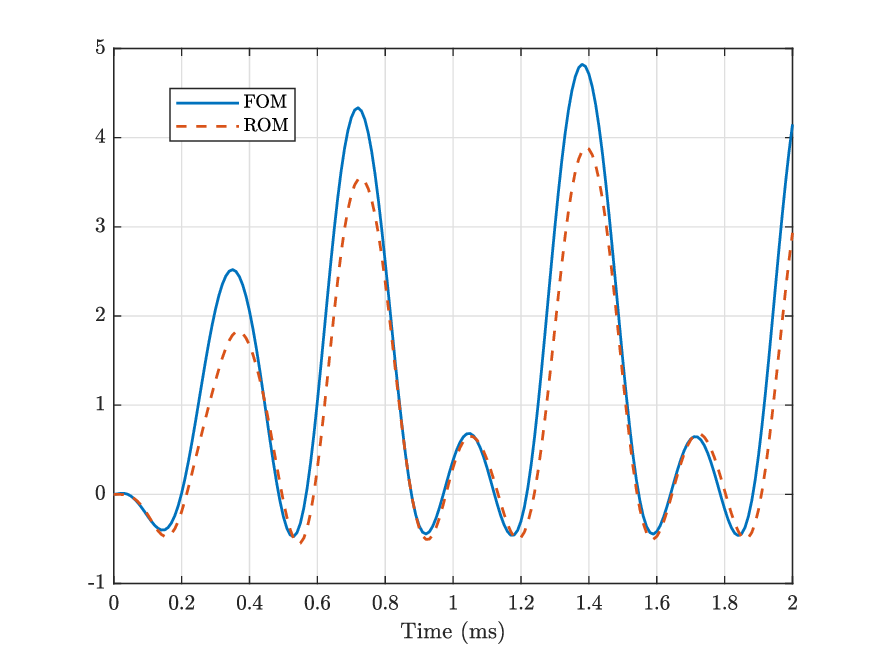}
		\caption{$\Theta = 2$ ($n_\phi = 30$).}
	\end{subfigure}
	\caption{Nonlinear ROM performance for $r = 4$ for different cluster size $\Theta$ and resulting number of monomials $n_\phi$. The plot shows the x-displacement of the point-of-interest marked in Figure~\ref{fig:pistonrod_undeformed}.}
	\label{fig:pistonrod_results}
\end{figure}

%% file: conclusion.tex
In this paper, a novel operator inference formulation for model reduction of nonlinear structural dynamics is developed.
By combining a sum-of-squares parametrization and a clustering-based sparse selection procedure, a flexible and efficient representation of models is developed.
Because of the approximation capabilities of polynomials, representative (reduced-order) models of dynamics governed by non-polynomial nonlinearities can be found.
This is illustrated by two numerical studies involving FEM models with nonlinear material laws and undergoing large deformation.
It was shown that the proposed method can effectively find ROMs satisfying stability properties, whereas a linear ROM or a ROM discovered without such stability constraints would not be accurate.

Although the method is promising, additional possibilities for improvement remain open for future work.
For example, the proposed hyperreduction scheme of Section \ref{sct:hyperreduction} could be further analysed to determine its theoretical properties.
Based on such an analysis, aspects such as stopping criteria or determining optimal cluster size could be investigated.
An additional point that deserves further research is the generalization to nonlinear structural dynamics involving non-constant mass or damping matrices (such as time-varying mass or velocity-dependent damping).
Such an extension would require finding suitable candidate Lyapunov functions that allows for an extension of the operator inference formulation in the convex setting.
Third, the stability analysis of Section \ref{sct:stabopinf} and subsequent embedding in the operator inference context could be extended to other notions of stability, such as incremental stability.